\documentclass[10pt]{amsart}
\usepackage{graphicx}
\usepackage{amscd}
\usepackage{amsmath}
\usepackage{amsthm}
\usepackage{amsfonts}
\usepackage{amssymb}
\usepackage{mathrsfs}
\usepackage{bm}
\usepackage{enumerate}
\usepackage{amsrefs}
\usepackage{xcolor}
\usepackage[colorlinks, citecolor=blue, linkcolor=red, pdfstartview=FitB]{hyperref}
\usepackage{comment}
\usepackage{hyperref}

\usepackage{tikz}
\usepackage{tikz-cd}
\usepackage{caption}
\usetikzlibrary{decorations.markings}
\usepackage{lipsum}
\usepackage{mathrsfs}

\usepackage{marginnote}	
\usepackage{soul} 

\newcommand{\bC}{{\mathbb{C}}}

\newcommand{\bF}{{\mathbb{F}}}

\newcommand{\bM}{{\mathbb{M}}}
\newcommand{\bN}{{\mathbb{N}}}

\newcommand{\bT}{{\mathbb{T}}}

\newcommand{\bZ}{{\mathbb{Z}}}

  \newcommand{\B}{{\mathcal{B}}}

  \newcommand{\F}{{\mathcal{F}}}
  
\renewcommand{\H}{{\mathcal{H}}}

  \newcommand{\K}{{\mathcal{K}}}  

  \newcommand{\M}{{\mathcal{M}}}

\renewcommand{\S}{{\mathcal{S}}}
  \newcommand{\T}{{\mathcal{T}}}

\newcommand{\rC}{\mathrm{C}}

\renewcommand{\phi}{\varphi}

\newcommand{\upchi}{{\raise.35ex\hbox{$\chi$}}}

\newcommand{\ol}{\overline}

\newcommand{\mycomment}[1]{}

\newcommand{\conv}{\operatorname{conv}}

\newtheorem{theorem}{Theorem}[section]
\newtheorem*{theorem*}{Theorem}
\newtheorem{lemma}[theorem]{Lemma}
\newtheorem*{lemma*}{Lemma}
\newtheorem{corollary}[theorem]{Corollary}

\newtheorem{proposition}[theorem]{Proposition}

\newtheorem{theoremx}{Theorem}

\date{\today}

\author{Ian Thompson}
\address{Department of Mathematical Sciences, University of Copenhagen, Universitetsparken 5, 2100 Copenhagen, Denmark}
\email{ian@math.ku.dk\vspace{-1ex}}



\thanks{The author was partially supported by an NSERC Postdoctoral Fellowship.}

\title[The universal commuting dilation constant]{On the universal commuting dilation constant}
\begin{document}
\begin{abstract}
    The universal commuting dilation constant $C_d$ is the smallest constant $\alpha$ such that every $d$-tuple of contractions dilates to a commuting $d$-tuple of normal operators with norm at most $\alpha$. The work of several authors shows that $1.5438 \lesssim C_2\leq 2$, and it has been asked on a few accounts whether $C_2<2$. We provide a positive answer that, in fact, produces a near optimal upper bound of $C_2 \leq \frac{2}{\sqrt{\phi}}$ where $\phi$ is the golden ratio. This tightens the gap on the universal commuting dilation constant to $1.5438 \lesssim C_2 \lesssim 1.5724$. We also tighten the known upper and lower bounds on $C_d$ for arbitrary $d$-tuples.
\end{abstract}
\maketitle

\section{Introduction}\label{s:intro}

One of the keystone results of operator theory is the Sz.-Nagy dilation theorem: a contractive (Hilbert space) operator admits a power dilation to a unitary operator \cite{sznagy1954contractions}. The Sz.-Nagy theorem lies at the epicentre of rich interactions between operator theory and complex function theory \cites{agler2002pick,nagy2010harmonic}. Since then, the implementation of dilation theoretic techniques has become ubiquitous to both operator theory and operator algebras. In addition, dilation theory offers a wide arsenal of tools that has had meaningful applications in quantum information theory \cites{chakraborty2019power,de2020quantum, hu2020quantum, manvcinska2024constant} and applied mathematics \cites{ben2002tractable, clouatre2025lifting, helton2019dilations, li2024simple}.

A leading problem in dilation theory today concerns dilation theorems that hold only up to a constant. The central question is not only to determine the existence of such constants, but also their optimal values. This problem has attracted significant attention in recent years \cites{gerhold2021dilations, gerhold2022dilations, gerhold2025empirical}, as it encodes inclusion problems for matrix convex sets or, equivalently, interpolation phenomena for completely positive maps \cites{passer2018minimal, passer2019shape, davidson2017dilations, fritz2017spectrahedral}. We remark that the idea itself has also seen increasing relevance in quantum information theory \cites{bluhm2018joint, bluhm2020compatibility, bluhm2023polytope, bluhm2025inclusion, bluhm2022incompatibility}. For our purposes, our principal interests lie in the work of Gerhold--Shalit and collaborators \cites{gerhold2021dilations, gerhold2022dilations, gerhold2025empirical}. Therein, numerous dilation theorems up to a constant have been discovered by using a variety of techniques ranging from random matrix theory, free probability theory, and mathematical physics.

We are concerned with a central problem in this direction: finding the true value of the so-called universal commuting dilation constant $C_d$. The constant $C_d$ is defined as the smallest value $\alpha$ so that every $d$-tuple of contractions admits a dilation to a $d$-tuple of commuting normal operators each with norm at most $\alpha$. This notion, in particular, has had applications in studying the joint measurability of quantum effects \cite{bluhm2018joint}. In the case of $d=1$, the dilation constant $C_d$ is equal to one and, in fact, the Sz.-Nagy theorem produces a power dilation with these properties. More broadly, it was shown in \cite[Section 7]{davidson2017dilations} that $C_d\leq d$ for every positive integer $d$. Due to the work of Passer and Passer--Shalit--Solel \cites{passer2018minimal, passer2019shape}, this was further improved to $\sqrt{d}\leq C_d \leq \sqrt{2d}$. By using pairs of certain $q$-commuting unitaries \cite[Section 7]{gerhold2022dilations}, it has also been shown that $1.5438\lesssim C_2$. On the other hand, determining an upper bound for $C_2$ that is stronger than the initial estimate \cite[Section 7]{davidson2017dilations} has remained completely elusive. This is despite an intensive treatment surrounding the problem \cites{gerhold2021dilations, gerhold2022dilations, passer2018minimal, passer2019shape}, including a recent empirical approach by estimating the corresponding commuting dilation constant for a pair of free Haar unitaries \cite{gerhold2025empirical}. Therein, it was suggested that $C_2\leq 2\sqrt{2/3}$ as this directly follows from combining the known \emph{optimal} constant for dilating free unitaries to Haar unitaries \cite[Corollary 3.8]{gerhold2021dilations} with the \emph{best possible} constant for dilating Haar unitaries to commuting unitaries. The main result of this work produces an upper bound on $C_2$ that is stronger than this and, in fact, is nearly optimal.

\begin{theoremx}\label{t:a}
    Given a pair of contractions $T_1, T_2\in B(\H)$ on a Hilbert space $\H$, there is a Hilbert space $\K\supseteq\H$ and a pair of commuting normal contractions $N_1, N_2\in B(\K)$ such that
    \[
        T_i =  P_\H \left(\frac{2}{\sqrt{\phi}} N_i \right)|_{\H}, \quad i=1,2,
    \]
    where $\phi$ denotes the golden ratio. In other words, the universal commuting dilation constant $C_2$ is bounded above by $\frac{2}{\sqrt{\phi}}$.
\end{theoremx}

Theorem \ref{t:a} then gives that $1.5438\lesssim C_2 \lesssim 1.5724$. For the strategy that we consider, the constant $\frac{2}{\sqrt{\phi}}$ was selected to be essentially the best possible conclusion relative to a few optimization arguments. We include the details in Appendix \ref{s:optimization}.

We then turn our attention to the value of the universal commuting dilation constant $C_d$ for general $d$-tuples of operators. First, we bootstrap the construction afforded to us by Theorem \ref{t:a} to obtain a recursive upper bound on $C_d$ (Theorem \ref{t:recursive-bound}). This allows us to conclude that $1.858\lesssim C_3\lesssim 2.025$, which improves upon the previous upper bound of $\sqrt{6}\approx 2.449$. Then, we exploit this recursion to obtain both a strict improvement to the upper bound on $C_d$ for general $d$-tuples of operators, as well as an asymptotic upper bound (Theorem \ref{t:asy-bound}). We are then naturally led to uncover new methods for finding stronger \emph{lower} bounds on $C_d$ for arbitrary $d$-tuples of operators. The collection of these results are summarized as follows.

\begin{theoremx}\label{t:b}
    The following statements hold.
    \begin{enumerate}[{\rm (i)}]
        \item For every positive integer $d\geq 2$, we have that $\sqrt{d}< C_d<\sqrt{2d}$.
        \item We have that $C_d \leq \sqrt{2d-\log d + O(1)}$ as $d\rightarrow\infty$.
    \end{enumerate}
\end{theoremx}

The best available lower bounds of $1.5438\lesssim C_2$ and $1.858\lesssim C_3$ arise from $q$--commuting unitaries (see \cite[Section 7]{gerhold2022dilations} and \cite[Corollary 4.7]{gerhold2021dilations}). Theorem \ref{t:uni-dil-lower} takes advantage of these two examples to return that, for every positive integer $d\geq 2$,
\[
    C_d \gtrsim 1.0916\sqrt{d} \qquad \text{or} \qquad C_d \gtrsim \sqrt{1.1916d - 0.1227}
\]
depending on when $d$ is even or odd, respectively. In both cases, this returns the same approximate lower bound of $1.0916\sqrt{d}$ for $d$ sufficiently large.

We now describe the organization of the manuscript. In Section \ref{s:prelim}, we record basic properties and notation. Within Section \ref{s:uni-comm-dil}, we prove our main results. In Subsection \ref{ss:pairs}, we prove Theorem \ref{t:a} and mention connections with the commuting dilation constant for free Haar unitaries. In Subsection \ref{ss:d-tuples}, we prove the recursive upper bound and the first assertion of Theorem \ref{t:b}. In Appendix \ref{s:optimization}, we offer supplementary results that describe the optimality of the proof strategy presented in Theorem \ref{t:a}. Finally, in Appendix \ref{s:recursion}, we prove some of the finer details required to prove the asymptotic upper bound given within the second statement of Theorem \ref{t:b}.

\section{Preliminaries}\label{s:prelim}

Let $A = (A_1,\ldots, A_d)\in B(\H)^d$ be a $d$-tuple of operators on a Hilbert space $\H$, and $B = (B_1,\ldots, B_d)\in B(\K)^d$ be a $d$-tuple of operators on a Hilbert space $\K$ with the property that $\K\supseteq\H$. Then, $A$ is said to be a \emph{compression} of $B$ if
\[
    P_\H B_i|_\H = A_i, \qquad i=1,\ldots, d,
\]
where $P_\H:\K\rightarrow\H$ denotes the orthogonal projection onto $\H$. In which case, we denote this by $A\prec B$. Alternatively, $B$ is also said to be a \emph{dilation} of $A$. One may consult \cite{paulsen2002completely} or \cite{shalit2021dilation} for detailed accounts on dilation theory.

Fix an integer $d\in\bN$. The \emph{universal commuting dilation constant}, denoted $C_d$, is the smallest constant $\alpha$ with the property that, for every $d$-tuple of contractions $A\in B(\H)^d$, there is a $d$-tuple of commuting normal operators $B\in B(\K)^d$ such that $A\prec \alpha B$. It is well-known that a $d$-tuple of contractions $A\in B(\H)^d$ dilates to a $d$-tuple of unitaries $U\in B(\H\oplus\H)^d$ where
\[
    U_i = \begin{bmatrix}
        A_i & (I-A_iA_i^*)^{1/2}\\
        (I-A_i^*A_i)^{1/2} & -A_i^*
    \end{bmatrix}, \qquad i=1,\ldots, d.
\]
From this, it can be shown that $C_d$ is the smallest constant $\alpha$ such that, for every $d$-tuple of unitaries $W\in B(\H)^d$, there is a commuting $d$-tuple of unitaries $U\in B(\H)^d$ such that $W\prec \alpha U$ \cite[Remark 1.2]{gerhold2021dilations}.

There have been several attempts at estimating the universal commuting dilation constants. It was first shown that $C_d \leq d$ for each integer $d$ (see \cite[Section 7]{davidson2017dilations} or \cite{helton2019dilations}). It was then proven that $C_d\leq \max\{ d, 2\sqrt{d}\}$ \cite[Corollary 6.11]{passer2018minimal} and, subsequently, that $C_d\leq \sqrt{2d}$ \cite[Theorem 4.4]{passer2019shape}. One also has a corresponding lower bound given by $C_d\geq \sqrt{d}$ \cite{passer2018minimal}. At the time of writing, the strongest bounds for arbitrary $d$ are given by $\sqrt{d}\leq C_d\leq \sqrt{2d}$. Further refinements on the lower bound for $C_d$ were shown when $d=2$ \cite[Section 7]{gerhold2022dilations} or $d=3$ \cite[Corollary 4.7]{gerhold2021dilations}.

We also consider a corresponding operator algebraic formulation on the value of this constant. For this, we introduce some standard facts as well as terminology. An \emph{operator system} is a self-adjoint unital subspace $\S\subset B(\H)$ where $\H$ is some Hilbert space. A linear map between operator systems $\Phi:\S\rightarrow\T$ is said to be \emph{positive} if $\Phi(s)\geq 0$ whenever $s\geq 0$. Further, $\Phi:\S\rightarrow\T$ is \emph{completely positive} if
\[
    \Phi^{(r)}:\bM_r(\S)\longrightarrow\bM_r(\T), \qquad [s_{ij}]\longmapsto [\Phi(s_{ij})],
\]
is positive for each integer $r\geq 1$. By Stinespring's dilation theorem \cite{stinespring1955positive}, if $\B$ is a unital $\rC^*$-algebra and $\Phi:\B\rightarrow B(\H)$ is a unital completely positive map, then there is a Hilbert space $\K\supseteq\H$ and a unital $*$-representation $\pi:\B\rightarrow B(\K)$ such that $\Phi = P_\H \pi(\cdot)|_\H$. Conversely, if $\pi:\B\rightarrow B(\K)$ is a unital $*$-representation and $\H\subseteq\K$ is a Hilbert space, then $\Phi:=P_\H \pi(\cdot)|_\H$ is a unital completely positive map.

Recall that the $\rC^*$-algebra $\rC(\bT^d)$ of continuous functions on the torus $\bT^d\subseteq\bC^d$ is the universal unital $\rC^*$-algebra generated by $d$ commuting unitaries, which are the coordinate functions $z_1,\ldots, z_d$. That is, whenever $U\in B(\H)^d$ is a $d$-tuple of commuting unitary operators on some Hilbert space $\H$, there is a unital $*$-homomorphism $\pi:\rC(\bT^d)\rightarrow \rC^*(U_1,\ldots, U_d)$ such that $\pi(z_i) = U_i$ for each $i=1,\ldots, d$. Let $\bF_d$ denote the free group on $d$ generators $a_1,\ldots, a_d$. Then, the (full) free group $\rC^*$-algebra $\rC^*(\bF_d)$ is the universal $\rC^*$-algebra generated by $d$ unitary operators. For $w\in\bF_d$, we let $\mathfrak{u}_w$ denote the corresponding unitary in $\rC^*(\bF_d)$. By Stinespring's dilation theorem, the inverse $C_d^{-1}$ of the universal commuting dilation constant is the largest constant $\alpha$ for which there is a unital completely positive map $\Phi:\rC(\bT^d) \rightarrow\rC^*(\bF_d)$ such that
\[
    \Phi(z_i) = \alpha \mathfrak{u}_{a_i}, \qquad i=1,\ldots, d.
\]
One may also regard $\rC(\bT^d)$ as the full group $\rC^*$-algebra $\rC^*(\bZ^d)$ of the free abelian group on $d$ generators \cite[Example 2.5.1]{brown2008textrm}. We let $e_1,\ldots, e_d\in\bZ^d$ denote the standard generators of $\bZ^d$. In this context, the universal commuting dilation constant admits another formulation:~ $C_d^{-1}$ is the largest constant $\alpha$ so that, for every $d$-tuple of unitary operators $U\in B(\H)^d$, there is a positive definite function $\Phi:\bZ^d \rightarrow B(\H)$ such that $\Phi(e_i) = \alpha U_i$ for every $i=1,\ldots, d$.

\section{The universal commuting dilation constant}\label{s:uni-comm-dil}

\subsection{Pairs of operators}\label{ss:pairs}

This subsection is devoted to constructing a unital completely positive map $\Phi:\rC(\bT^2)\rightarrow\rC^*(\bF_2)$ such that the coordinate functions $z,w\in\rC(\bT^2)$ are mapped to
\[
\Phi(z) = \frac{\sqrt{\phi}}{2}\mathfrak{u}_a, \qquad \Phi(w) = \frac{\sqrt{\phi}}{2}\mathfrak{u}_b
\]
where $\phi$ denotes the golden ratio and $a,b\in\bF_2$ are the canonical generators (Theorem \ref{t:uni-comm-dil-constant}). From here, we will readily have that $C_2\leq \frac{2}{\sqrt{\phi}}$. To accomplish this, we construct a pair of commuting unitary operators $Z$ and $W$, and then $\Phi$ will arise out of a particular compression of the pair.

\begin{proposition}\label{p:cocyc-uni-comm-uni}
Let $\K$ be a Hilbert space and $S\in B(\K)$ be a unitary operator. Assume that $\rC^*(\bF_2)\subseteq B(\H)$ for some Hilbert space $\H$ and consider the pair of operators $Z,W\in B(\H\otimes \ell^2(\bZ)\otimes\K)$ defined by
    \[
        Z(h\otimes\delta_m\otimes\eta) = \mathfrak{u}_ah\otimes\delta_{m+1}\otimes\eta, \qquad W(h\otimes\delta_m\otimes\eta) = \mathfrak{u}_{a^m ba^{-m}}h\otimes\delta_m\otimes S\eta,
    \]
for $h\in\H$, $m\in\bZ$, and $\eta\in\K$. Then, $Z$ and $W$ are commuting unitary operators.
\end{proposition}

\begin{proof}
    It is clear that both $Z$ and $W$ are unitary operators. Now, observe that
    \[
        ZW(h\otimes\delta_m\otimes\eta) = Z(\mathfrak{u}_{a^m ba^{-m}}h\otimes\delta_m\otimes S\eta) = \mathfrak{u}_{a^{m+1}ba^{-m}}h \otimes \delta_{m+1}\otimes S\eta
    \]
    and, likewise, one finds that
    \[
        WZ(h\otimes\delta_m\otimes\eta) = W(\mathfrak{u}_a  h \otimes\delta_{m+1}\otimes \eta) = \mathfrak{u}_{(a^{m+1}ba^{-(m+1)})a}h \otimes \delta_{m+1}\otimes S\eta.
    \]
    So, indeed, $Z$ and $W$ are commuting unitaries.
\end{proof}

We remark that a similar shifting-and-twisting dilation construction appeared in \cite[Lemma 3.1]{gerhold2024dilation}. The author thanks Orr Shalit for bringing this to our attention.

For such pairs, we construct unital completely positive maps into the free group $\rC^*$-algebra by considering compressions of the following form.

\begin{proposition}\label{p:univ-uni-compress}
    Let $\K$ be a Hilbert space and $S\in B(\K)$ be some unitary operator. Assume that $\rC^*(\bF_2)\subseteq B(\H)$ for some Hilbert space $\H$ and consider the pair of commuting unitaries $Z,W\in B(\H\otimes \ell^2(\bZ)\otimes\K)$ defined by
    \[
        Z(h\otimes\delta_m\otimes\eta) = \mathfrak{u}_ah\otimes\delta_{m+1}\otimes\eta, \qquad W(h\otimes\delta_m\otimes\eta) = \mathfrak{u}_{a^m ba^{-m}}h\otimes\delta_m\otimes S\eta,
    \]
    for $h\in\H, m\in\bZ$, and $\eta\in K$. Assume that $\xi\in\ell^2(\bZ)\otimes\K$ is a unit vector given by $\xi=\sum_{m\in\bZ}\delta_m\otimes\xi_m$ for some sequence $(\xi_m)_m$ in $\K$, and define an isometry $V:\H\rightarrow \H\otimes \ell^2(\bZ)\otimes\K$ by $Vh=h\otimes\xi$. Then, the following statements hold.
    \begin{enumerate}[{\rm (i)}]
        \item We have that
        \[
            V^*ZV = \left(\sum_{m\in\bZ}\langle \xi_m,\xi_{m+1}\rangle \right)\mathfrak{u}_a \quad \text{and} \quad V^*WV = \sum_{m\in\bZ} \langle S\xi_m,\xi_m\rangle\mathfrak{u}_{a^m b a^{-m}}.
        \]
        \item We have that $V^* \rC^*(Z,W) V \subseteq \rC^*(\bF_2)$.
    \end{enumerate}
\end{proposition}

\begin{proof}
    (i): Fix $h\in\H$. First, note that
    \[
        V h = h\otimes\xi = \sum_{m\in\bZ} h\otimes\delta_m\otimes\xi_m,
    \]
    and so
    \[
        V^*ZVh = V^*\left(\sum_{m\in\bZ}\mathfrak{u}_ah\otimes\delta_{m+1}\otimes\xi_m\right) = \left(\sum_{m\in\bZ} \langle \xi_m,\xi_{m+1}\rangle\,\right) \mathfrak{u}_ah.
    \]
    Therefore
    \[
        V^*ZV = \left( \sum_{m\in\bZ} \langle \xi_m,\xi_{m+1}\rangle \right)\mathfrak{u}_a.
    \]
    Similarly, we see that
    \[
        V^*WV h = V^*\left(\sum_{m\in\bZ} \mathfrak{u}_{a^m b a^{-m}}h\otimes\delta_m\otimes S\xi_m\right) = \sum_{m\in\bZ} \langle S\xi_m,\xi_m\rangle\, \mathfrak{u}_{a^m b a^{-m}}h.
    \]
    Thus, we obtain
    \[
        V^*WV = \sum_{m\in\bZ} \langle S\xi_m,\xi_m\rangle \mathfrak{u}_{a^m b a^{-m}}.
    \]

    (ii): Fix $r,s\in\bZ$. Then, we have that
    \begin{align*}
        V^* Z^r W^s Vh & = V^*\left( \sum_{m\in\bZ} \mathfrak{u}_{a^r a^mb^sa^{-m}}h \otimes \delta_{m+r}\otimes(S^s\xi_m)\right)\\
        & = \sum_{m\in\bZ} \langle S^s \xi_m, \xi_{m+r}\rangle \mathfrak{u}_{a^r(a^mb^sa^{-m})}h.
    \end{align*}
    Therefore,
    \[
        V^* Z^rW^s V = \sum_{m\in\bZ} \langle S^s\xi_m, \xi_{m+r}\rangle \mathfrak{u}_{a^{r+m}b^sa^{-m}},
    \]
    and the series is easily seen to converge in the norm topology. Thus, we have that $V^* Z^rW^s V\in\rC^*(\bF_2)$ and, from which, it follows that $V^*\rC^*(Z,W)V\subseteq\rC^*(\bF_2)$.
\end{proof}

By Propositions \ref{p:cocyc-uni-comm-uni} and \ref{p:univ-uni-compress}, we may now derive an upper bound on the commuting dilation constant for a pair of free unitaries. Within Appendix \ref{s:optimization}, we justify the strength behind our choice of parameters.

\begin{theorem}\label{t:uni-comm-dil-constant}
    There is a unital completely positive map $\Phi:\rC(\bT^2)\rightarrow\rC^*(\bF_2)$ such that $\Phi(z) = \frac{\sqrt{\phi}}{2}\mathfrak{u}_a$ and $\Phi(w) = \frac{\sqrt{\phi}}{2}\mathfrak{u}_b$.
\end{theorem}

\begin{proof}
    Throughout, assume that $\rC^*(\bF_2)\subseteq B(\H)$ for some Hilbert space $\H$. Let $e_1, e_2$ be the standard orthonormal basis of $\bC^2$ and consider the unitary operator $S\in \bM_2$ defined by $Se_1=e_1$ and $Se_2=-e_2$.
    
    First, we construct an appropriate unit vector $\xi\in\ell^2(\bZ)\otimes\bC^2$ of the form $\sum_{m\in\bZ} \delta_m\otimes\xi_m$ for which to apply Proposition \ref{p:univ-uni-compress}. To this end, we consider the following constants
    \[
        p=\frac{5+\sqrt5}{10} \quad \text{and} \quad r=\sqrt{\frac{2(\sqrt5-1)}{5}},
    \]
    and note that $0<r<p<1$. Now, set
    \[
        \xi_0 = \sqrt{\frac{p+r}{2}}e_1 + \sqrt{\frac{p-r}{2}}e_2,
    \]
    and observe that $\|\xi_0\|^2 = p$. Furthermore, define
    \[
        q=\frac{1-p}{2}, \quad \rho=\sqrt{\frac{q}{\frac{1}{2}(p+\sqrt{p^2-r^2})+q}}, \quad \text{and} \quad \alpha_m = \sqrt{q(1-\rho^2)}\,\rho^{m-1}, \ \ \ m\geq 1.
    \]
    Then, we consider
    \[
        \xi=\sum_{m\in\bZ}\delta_m\otimes \xi_m\in \ell^2(\bZ)\otimes\bC^2
    \]
    where $\xi_m = \alpha_{|m|}\frac{e_1+e_2}{\sqrt2}$ for each $m\neq 0$. Note that $\xi$ is a unit vector since
    \[
        \|\xi\|^2 = \|\xi_0\|^2 +2\sum_{m\geq1}\alpha_m^2 = p + 2q(1-\rho^2)\sum_{m\geq1}\rho^{2(m-1)} = p+2q = 1.
    \]

    By Proposition \ref{p:cocyc-uni-comm-uni}, the pair of operators on $\H\otimes\ell^2(\bZ)\otimes\bC^2$ defined by
    \[
        Z(h\otimes\delta_m\otimes\eta) = \mathfrak{u}_ah\otimes\delta_{m+1}\otimes\eta \quad \text{and} \quad W(h\otimes\delta_m\otimes\eta) = \mathfrak{u}_{a^mba^{-m}}h\otimes\delta_m\otimes S\eta
    \]
    are commuting unitaries. Thus, there is a unital $*$-representation $\pi:\rC(\bT^2)\rightarrow B(\H\otimes\ell^2(\bZ)\otimes\bC^2)$ such that $\pi(z) = Z$ and $\pi(w) = W$. Let $V:\H\rightarrow\H\otimes\ell^2(\bZ)\otimes\bC^2$ be the isometry given by $Vh = h\otimes\xi$ and define a unital completely positive map
    \[
        \Phi_1:\rC(\bT^2)\longrightarrow B(\H), \qquad f\longmapsto V^*\pi(f)V.
    \]
    By Proposition \ref{p:univ-uni-compress}, we have that $\Phi_1(\rC(\bT^2))\subseteq\rC^*(\bF_2)$ and that
    \[
        \Phi_1(z) = \left( \sum_{m\in\mathbb Z} \langle \xi_m,\xi_{m+1}\rangle\right)\mathfrak{u}_a \quad \text{and} \quad \Phi_1(w) = \sum_{m\in\mathbb Z}\langle S\xi_m,\xi_m\rangle\mathfrak{u}_{a^m b a^{-m}}.
    \]
    For each $m\neq 0$, note that $\langle S\xi_m,\xi_m\rangle = 0$ and so,
    \[
        \Phi_1(w)= \langle S\xi_0,\xi_0\rangle\mathfrak{u}_b =r\mathfrak{u}_b.
    \]
    Setting $A = \langle \xi_0,\frac{e_1+e_2}{\sqrt{2}}\rangle$, we then find that
    \begin{align*}
        \sum_{m\in\bZ}\langle\xi_m, \xi_{m+1}\rangle & =2A\alpha_1  +  2\sum_{m\geq 1}\langle \xi_m, \xi_{m+1}\rangle\\
        & = 2A\alpha_1 + 2 \sum_{m\geq 1}\alpha_m \alpha_{m+1}\\
        & = 2A\sqrt{\frac{A^2q}{A^2+q}} +2q \sqrt{\frac{q}{A^2+q}}\\
        & = 2\sqrt{q\left(A^2 + q \right)} = \sqrt{\frac{1+\sqrt{5}}{10}}.
    \end{align*}
    Therefore, $\Phi_1(z)=\widetilde{r} \mathfrak{u}_a$ where $\widetilde{r}:=\sqrt{\frac{1+\sqrt{5}}{10}}$.
    
    Let $\beta$ be the $*$-automorphism of $\rC^*(\bF_2)$ induced by the group automorphism of $\bF_2$ that exchanges $a$ and $b$. Similarly, let $\sigma$ be the $*$-automorphism of $\rC(\bT^2)$ that exchanges the coordinate functions $z$ and $w$. Define a unital completely positive $\Phi_2:\rC(\bT^2)\rightarrow\rC^*(\bF_2)$ by $\Phi_2 =\beta\circ \Phi_1\circ \sigma$. Furthermore, define a unital completely positive map $\Phi:\rC(\bT^2)\rightarrow\rC^*(\bF_2)$ by $\Phi=\frac{\Phi_1+\Phi_2}{2}$. Then, $\Phi$ satisfies
    \[
        \Phi(z)=\frac{r+\widetilde{r}}{2}\mathfrak{u}_a \quad \text{and} \quad \Phi(w)=\frac{r+\widetilde{r}}{2}\mathfrak{u}_b.
    \]
    Finally, observe that
    \[
        \frac{r+\widetilde{r}}{2} = \frac{1}{2}\left(\sqrt{\frac{2(\sqrt{5}-1)}{5}} + \sqrt{\frac{1+\sqrt{5}}{10}} \right) = \frac{1}{4}\sqrt{2(1+\sqrt{5})} = \frac{\sqrt{\phi}}{2}.
    \]
\end{proof}

As a direct consequence, we obtain the corresponding dilation theorem.

\begin{corollary}\label{c:uni-comm-dilation}
    The universal commuting dilation constant is bounded above by $\frac{2}{\sqrt{\phi}}$. In other words, given contractions $T_1, T_2\in B(\H)$, there is a Hilbert space $\K\supseteq \H$ and a pair of commuting normal contractions $N_1,N_2\in B(\K)$ such that
    \[
        T_i =  P_\H \left(\left.\frac{2}{\sqrt{\phi}}N_i\right)\right|_\H, \quad i=1,2.
    \]
\end{corollary}

\begin{proof}
    By Theorem \ref{t:uni-comm-dil-constant}, there is a unital completely positive map $\Phi:\rC(\bT^2)\rightarrow\rC^*(\bF_2)$ such that
    \[
        \Phi(z) = \frac{\sqrt{\phi}}{2} \mathfrak{u}_a \quad \text{and} \quad \Phi(w) = \frac{\sqrt{\phi}}{2}\mathfrak{u}_b.
    \]
    By Stinespring's dilation theorem \cite{stinespring1955positive}, we see that $(\mathfrak{u}_a,\mathfrak{u}_b)\prec \frac{2}{\sqrt{\phi}}(z,w)$. By \cite[Remark 1.2]{gerhold2021dilations}, we can therefore conclude that $C_2 \leq \frac{2}{\sqrt{\varphi}}$, and the conclusion follows.
\end{proof}

Corollary \ref{c:uni-comm-dilation} also improves upon the current upper bound on the commuting dilation constant for a pair of free Haar unitaries  \cites{gerhold2021dilations, gerhold2025empirical}. For this, let $\lambda:\bF_d\rightarrow B(\ell^2(\bF_d))$ denote the left regular representation of $\bF_d$ and $\rC^*_\lambda(\bF_d)$ denote the $\rC^*$-algebra generated by the image of $\lambda$. Define a constant
\[
    H_d:=\inf\{\alpha : (\lambda_{a_1},\ldots,\lambda_{a_d})\prec \alpha(z_1,\ldots, z_d)\}.
\]
In previous work, this has also been denoted by $c_{f,0}^{(d)}$ and $c(u_f, u_0)$ (see \cite{gerhold2021dilations} and \cite{gerhold2025empirical}, respectively). By appealing to free probabilistic techniques and the geometry of free spectrahedra \cite[Theorems 3.9 and 3.10]{gerhold2021dilations}, it was established that
\[
    2\sqrt{1-\frac{1}{d}} \leq H_d \leq 2 \sqrt{1-\frac{1}{2d}}.
\]
In the case of $d=2$, we have that $\sqrt{2}\leq H_2 \leq \sqrt{3}$. Moreover, empirical methods appear to suggest that one should anticipate that the true value of $H_2$ is $\sqrt{2}$ \cite{gerhold2025empirical}. While we are unable to determine whether this is the case, our main theorem still tightens the current upper bound.

\begin{corollary}\label{c:haar-comm-dil-constant}
    The commuting dilation constant for a pair of free Haar unitaries $H_2$ is bounded above by $\frac{2}{\sqrt{\phi}}$.
\end{corollary}

\begin{proof}
    Let $q:\rC^*(\bF_2)\rightarrow \rC^*_\lambda(\bF_2)$ denote the canonical surjective $*$-homomorphism. By Theorem \ref{t:uni-comm-dil-constant}, there is a unital completely positive map $\Phi:\rC(\bT^2)\rightarrow\rC^*(\bF_2)$ such that $\Phi(z) = \frac{\sqrt{\phi}}{2}\mathfrak{u}_a$ and $\Phi(w) = \frac{\sqrt{\phi}}{2}\mathfrak{u}_b$. Applying Stinespring's dilation theorem to the unital completely positive map $q\circ\Phi:\rC(\bT^2)\rightarrow\rC^*_\lambda(\bF_2)$, we can then conclude that $(\lambda_a,\lambda_b)\prec \frac{2}{\sqrt{\phi}}(z,w)$.
\end{proof}

Surprisingly, the constant $\frac{2}{\sqrt{\phi}}$ may also be appearing in the empirical data surrounding the computation of $H_2$ \cite[Section 4]{gerhold2025empirical}. To see this, let $U_1^{(N)}, U_2^{(N)}\subseteq\bM_N$ be a pair of independent unitaries that are sampled with respect to Haar measure on the unitary group of $N\times N$ matrices. Then, consider the random variable
\[
    c_2(N) =\inf\left\{ \alpha : (U_1^{(N)},U_2^{(N)})\prec\alpha(z,w)\right\}.
\]
The random variable $c_2(N)$ was used in \cite[Section 4]{gerhold2025empirical} to empirically test for the true value of $H_2$. Therein, no samples of $c_2(N)$ were seen to exceed $\frac{\sqrt{2}+\sqrt{3}}{2}\approx1.5731$. It is then perhaps conceivable that $\frac{2}{\sqrt{\phi}}\approx 1.5724$ could be the true value for $C_2$, although we have no innate reason to believe this either.

\subsection{Arbitrary d-tuples of operators}\label{ss:d-tuples}

We now provide a broader outlook into the true value of the universal commuting dilation constant for arbitrary $d$-tuples of operators. For this, recall that the strongest general bounds are that $\sqrt{d}\leq C_d \leq\sqrt{2d}$ for every integer $d\geq 2$. First, we bootstrap the construction in Theorem \ref{t:uni-comm-dil-constant} to gain access to a recursive upper bound on $C_d$ for arbitrary $d$-tuples. Asymptotically, this gives an upper bound of $C_d\leq \sqrt{2d-\log d+O(1)}$ as $d\rightarrow\infty$. After this, we prove a new lower bound on $C_d$ for every integer $d\geq4$.

We start by proving the recursive upper bound for $C_d$. In essence, this follows from the construction of Subsection \ref{ss:pairs} where we now treat the operator $W$ in the pair $(Z,W)$ as a $d$-tuple of operators, and then apply an averaging trick.

\begin{theorem}\label{t:recursive-bound}
    Fix a positive integer $d\geq 1$. For any constants $0 < r \leq p < 1$, we have that
    \[
        C_{d+1}\leq \frac{d+1}{A(p,r) + drC_d^{-1}}
    \]
    where
    \[
        A(p,r) = \sqrt{(1-p)\left( 1+\sqrt{p^2-r^2}\right)}.
    \]
\end{theorem}

\begin{proof}
    Consider the vectors $v,\xi_0\in\bC^2$ given by
    \[
        v = \frac{e_1+e_2}{\sqrt{2}} \qquad \text{and} \qquad \xi_0 = \sqrt{\frac{p+r}{2}} e_1 + \sqrt{\frac{p-r}{2}}e_2.
    \]
    In addition, let $S\in B(\bC^2)$ be the unitary defined by $Se_1 = e_1$ and $Se_2 = -e_2$. Setting $q = \frac{1-p}{2}$, we define
    \[
        \rho = \sqrt{\frac{q}{\langle\xi_0, v\rangle^2+q}}, \qquad \alpha_m = \sqrt{q(1-\rho^2)}\rho^{m-1}, \qquad m\geq 1,
    \]
    and set $\xi_m = \alpha_{|m|}v$ for each non-zero $m\in\bZ$. As in Theorem \ref{t:uni-comm-dil-constant}, we have that $\xi = \sum_{m\in\bZ}\delta_m\otimes\xi_m$ is a unit vector in $\ell^2(\bZ)\otimes\bC^2$ that satisfies
    \[
        \sum_{m\in\bZ}\langle\xi_m, \xi_{m+1}\rangle = A(p,r), \qquad \langle S\xi_m, \xi_m\rangle = \delta_{m,0}r.
    \]

    Next, we construct a $(d+1)$-tuple of commuting unitaries. First, represent $\rC^*(\bF_{d+1})$ faithfully on some Hilbert space $\H$. Fix $\varepsilon>0$. Then, there is $\tau>C_d^{-1}-\varepsilon$ and a unital completely positive map $\psi:\rC(\bT^d)\rightarrow\rC^*(\mathfrak{u}_{a_2}, \ldots, \mathfrak{u}_{a_{d+1}})$ such that $\psi(z_j) = \tau\mathfrak{u}_{a_j}$ for each $2\leq j\leq d+1$. By Stinespring's dilation theorem \cite{stinespring1955positive}, we then obtain some Hilbert space $\M$, a $d$-tuple $(Y_2,\ldots, Y_{d+1})$ of commuting unitary operators on $\H\otimes\M$ and a unit vector $\omega\in\M$ such that, if $Q:\H\rightarrow\H\otimes\M$ denotes the isometry defined by $Q h = h\otimes \omega$, then we have that
    \[
        Q^* Y_j Q = \tau\mathfrak{u}_{a_j}, \qquad 2\leq j\leq d+1,
    \]
    and that $Q^*\rC^*(Y_2, \ldots, Y_{d+1})Q \subseteq \rC^*(\mathfrak{u}_{a_2},\ldots, \mathfrak{u}_{a_{d+1}})$. In turn, consider the unitary operator $U := \mathfrak{u}_{a_1}\otimes I_\M$ on $\H\otimes\M$ and define operators $Z_1,\ldots, Z_{d+1}$ on $\K = (\H\otimes\M)\otimes\ell^2(\bZ)\otimes\bC^2$ by
    \[
        Z_1(x\otimes\delta_m \otimes\eta) = Ux\otimes \delta_{m+1}\otimes\eta, \qquad Z_j(x\otimes \delta_m\otimes\eta) = U^mY_jU^{-m}x\otimes \delta_m \otimes S\eta,
    \]
    where $2\leq j\leq d+1$. It is straightforward to see that $Z_2,\ldots, Z_{d+1}$ are pairwise commuting. Furthermore, from Proposition \ref{p:cocyc-uni-comm-uni}, it follows that $(Z_1, \ldots, Z_{d+1})$ is a $(d+1)$-tuple of commuting unitary operators.

    We now construct a unital completely positive map $\Psi_1:\rC(\bT^{d+1})\rightarrow \rC^*(\bF_{d+1})$ from the $(d+1)$-tuple $(Z_1,\ldots, Z_{d+1})$. For this, consider the isometry $V:\H\rightarrow\K$ defined by $Vh = \sum_{m\in\bZ} h\otimes\omega\otimes\delta_m\otimes\xi_m$. Then, we have that
    \[
        V^*Z_1V = \left(\sum_{m\in\bZ}\langle\xi_m, \xi_{m+1}\rangle\right)\mathfrak{u}_{a_1} = A(p,r)\mathfrak{u}_{a_1}
    \]
    and that, for $2\leq j\leq d+1$,
    \[
        V^*Z_j V = \sum_{m\in\bZ} \langle S\xi_m, \xi_m\rangle Q^* U^m Y_jU^{-m}Q = rQ^* Y_j Q = r\tau\mathfrak{u}_{a_j}.
    \]
    Since $Z_1,\ldots, Z_{d+1}$ are commuting unitary operators, there exists a unital completely positive map $\Psi_1:\rC(\bT^{d+1})\rightarrow B(\H)$ such that $\Psi_1(z_j) = V^*Z_j V$ for each $1\leq j\leq d+1$. Furthermore, similar to Proposition \ref{p:univ-uni-compress}, one has that
    \[
        V^* Z_1^{s_1}\ldots Z_{d+1}^{s_{d+1}}V = \sum_{m\in\bZ} \langle S^{s_2+\ldots +s_{d+1}}\xi_m, \xi_{m+s_1}\rangle \mathfrak{u}_{a_1}^{m+s_1} Q^* Y_2^{s_2}\ldots Y_{d+1}^{s_{d+1}} Q \mathfrak{u}_{a_1}^{-m}
    \]
    for any $s_1,\ldots, s_{d+1}\in\bZ$. Therefore $\Psi_1(\rC(\bT^{d+1}))\subseteq\rC^*(\bF_{d+1})$, as desired.

    We now have a unital completely positive map $\Psi_1: \rC(\bT^{d+1})\rightarrow\rC^*(\bF_{d+1})$ with
    \[
        \Psi_1(z_1) = A(p,r)\mathfrak{u}_{a_1}, \quad \Psi_1(z_j) = r\tau\mathfrak{u}_{a_j}, \quad 2\leq j\leq d+1.
    \]
    Similarly, for each $1\leq i \leq d+1$, there exists a unital completely positive map $\Psi_i:\rC(\bT_{d+1})\rightarrow\rC^*(\bF_{d+1})$ such that
    \[
        \Psi_i(z_j) = \begin{cases}
            A(p,r)\mathfrak{u}_{a_j}, & j=i,\\
            r\tau\mathfrak{u}_{a_j}, & j\neq i.
        \end{cases}
    \]
    In turn, we define a unital completely positive map $\Psi:\rC(\bT^{d+1})\rightarrow \rC^*(\bF_{d+1})$ by $\Psi = \frac{1}{d+1} \sum_{i=1}^{d+1}\Psi_i$. Then, we have that
    \[
        \Psi(z_j) = \frac{A(p,r) + dr \tau}{d+1}\mathfrak{u}_{a_j}, \qquad 1\leq j\leq d+1.
    \]
    Consequently, we have that
    \[
        C_{d+1} \leq \frac{d+1}{A(p,r) + dr \tau}.
    \]
    Finally, as $\tau>C_d^{-1}-\varepsilon$ and $\varepsilon>0$ was arbitrary, it follows that
    \[
        C_{d+1} \leq \frac{d+1}{A(p,r) + dr C_d^{-1}}.
    \]
\end{proof}

Theorem \ref{t:recursive-bound} improves upon the available upper bounds on $C_d$, especially for small values of $d$. Indeed, the strongest known bounds for $C_3$ are that $1.858\lesssim C_3 \lesssim 2.449$ (see \cite[Corollary 4.7]{gerhold2021dilations} and \cite[Theorem 4.4]{passer2019shape}). We highlight the degree to which Theorem \ref{t:recursive-bound} improves upon this particular upper bound.

\begin{corollary}\label{c:triple-upper}
    The universal commuting dilation constant $C_3$ is bounded above by $ 3\sqrt{\frac{2}{1+\sqrt{7+2\sqrt{5}}}}\approx 2.025$.
\end{corollary}

\begin{proof}
    By Theorems \ref{t:uni-comm-dil-constant} and \ref{t:recursive-bound}, it follows that
    \[
        C_3 \leq \inf_{0<r\leq p<1}\frac{3}{A(p,r) + r\sqrt{\phi}}.
    \]
    By a straightforward optimization argument, this upper bound is readily seen to simplify to
    \[
        C_3 \leq 3\sqrt{\frac{2}{1+\sqrt{7+2\sqrt{5}}}}.
    \]
\end{proof}

In addition, Theorem \ref{t:recursive-bound} gives greater access to large-scale control of the universal commuting dilation constant. Unfortunately, part of this requires a detailed estimate that distracts from the discussion at hand. As such, we choose to defer some of the finer details until Appendix \ref{s:recursion}. The author thanks Adam Dor-On for helpful contributions concerning these finer details.

\begin{theorem}\label{t:asy-bound}
    The following statements hold.
    \begin{enumerate}[{\rm (i)}]
        \item For every positive integer $d\geq1$, we have that $C_d<\sqrt{2d}$.
        \item We have that $C_d$ is bounded above by $\sqrt{2d-\log d + O(1)}$ as $d\rightarrow\infty$.
    \end{enumerate}
\end{theorem}

\begin{proof}
    (i): The statement follows from an inductive argument. Indeed, assuming that $C_d<\sqrt{2d}$, Theorem \ref{t:recursive-bound} implies that
    \[
        C_{d+1}\leq \frac{d+1}{A\left(1-\frac{1}{2d},1-\frac{1}{2d}\right) + d\left(1-\frac{1}{2d}\right)C_d^{-1}}.
    \]
    Then note that $A\left(1-\frac{1}{2d},1-\frac{1}{2d}\right) = \frac{1}{\sqrt{2d}}$ and so,
    \begin{align*}
        A\left(1-\frac{1}{2d},1-\frac{1}{2d}\right) + d\left(1-\frac{1}{2d} \right)C_d^{-1}  & > \frac{1}{\sqrt{2d}} + d\left(1-\frac{1}{2d}\right)\frac{1}{\sqrt{2d}}\\
        & = \frac{2d+1}{2\sqrt{2d}} > \sqrt{\frac{d+1}{2}}.
    \end{align*}
    Thus, we obtain that $C_{d+1}<\sqrt{2(d+1)}$.

    (ii): For each $0\leq r\leq 1$, set $M(r) = \sup_{p\leq r\leq 1} A(p,r)$. Then, recursively define a sequence of constants by
    \[
        \gamma_{d+1} = \sup_{0\leq r\leq 1}\frac{M(r)+dr\gamma_d}{d+1}, \qquad \gamma_1=1.
    \]
    By Theorem \ref{t:recursive-bound}, we have that $C_d\leq \frac{1}{\gamma_d}$. In Proposition \ref{p:asy-3} of Appendix \ref{s:recursion}, it is found that
    \[
        \frac{1}{\gamma_d} = \sqrt{2d-\log d+O(1)} \qquad \text{as $d\rightarrow\infty$},
    \]
    from which the conclusion follows.
\end{proof}

Unfortunately, tight upper bounds on the universal commuting dilation constant still appear somewhat out of reach for arbitrary $d$-tuples. Indeed, the recursive upper bound in Theorem \ref{t:recursive-bound} is ultimately tied to the construction afforded by Subsection \ref{ss:pairs}. However, as Theorem \ref{t:asy-bound} illustrates, bootstrapping this construction does not necessarily retain a tight control on $C_d$ as $d$ becomes large.

Accordingly, we now extract new lower bounds associated to $C_d$ for arbitrary $d\geq 4$. In the case of $d=2$ and $d=3$, the strongest known lower bound was found by considering $q$-commuting unitaries (see \cite[Section 7]{gerhold2022dilations} and \cite[Corollary 4.7]{gerhold2021dilations}). We prove a lower bound on the universal commuting dilation constant for arbitrary $d$-tuples through an explicit construction involving pairs of $q$-commuting unitaries.

For this, we recall some relevant background from \cites{gerhold2021dilations,gerhold2022dilations}. Given a real anti-symmetric matrix $\Theta = [\theta_{k,l}]\in\bM_d$, we let $w(\Theta) = (w_1(\Theta),\ldots, w_d(\Theta))$ denote the $d$-tuple of unitaries satisfying $w_l(\Theta)w_k(\Theta) = e^{i\theta_{k,l}}w_k(\Theta) w_l(\Theta)$ and such that, whenever $v = (v_1,\ldots, v_d)$ is some other $d$-tuple of unitaries such that $v_lv_k = e^{i\theta_{k,l}}v_k v_l$, then there is a unital surjective $*$-homomorphism
\[
    \sigma: \rC^*(w_1(\Theta),\ldots, w_d(\Theta))\longrightarrow\rC^*(v_1,\ldots, v_d), \qquad w_i(\Theta)\longmapsto v_i.
\]
When the context is clear, we will suppress $\Theta$ and simply write $(w_1, \ldots, w_d)$. As in \cite{gerhold2021dilations}, we let $c_\Theta$ denote the smallest constant such that $w(\Theta)\prec c_\Theta (z_1, \ldots, z_d)$ where $z_1,\ldots, z_d\in\rC(\bT^d)$ are the coordinate functions. In particular, one has that $C_d\geq c_\Theta$. In \cite[Theorem 4.4]{gerhold2021dilations}, it was shown that
\[
    c_\Theta = \frac{1}{\inf\{\|\text{Re} X\| : X\in \text{conv}(w(\Theta))\}}.
\]

Given a real number $\theta$, we let $c_\theta$ denote the constant $c_\Theta$ corresponding to the $2\times 2$ anti-symmetric matrix $\Theta$ with zero diagonal and $(1,2)$-entry equal to $\theta$. In which case, we simply denote the pair of operators $(w_1(\Theta), w_2(\Theta))$ by $(w_1(\theta), w_2(\theta))$. We remark that, in this setting, the constant $c_\theta$ coincides with the constant from \cite{gerhold2022dilations}. A priori, the constant $c_\theta$ only gives access to lower bounds on $C_2$. However, through the use of an explicit matrix construction, one can upgrade this to new lower bounds on the universal commuting dilation constant for arbitrary $d$-tuples of operators.

\begin{theorem}\label{t:uni-dil-lower}
    Let $d\geq 2$ be a positive integer and express $d=\sum_{j=1}^k d_j$ as a sum of positive integers $d_j \geq 2$. If $\Theta_1, \ldots,\Theta_k$ are real anti-symmetric matrices where $\Theta_j$ is of size $d_j\times d_j$, then we have that
    \[
        C_d \geq \sqrt{\sum_{j=1}^k c_{\Theta_j}^2}.
    \]
\end{theorem}

\begin{proof}
    Define a real anti-symmetric matrix $\Theta$ with diagonal blocks $\Theta_1, \ldots, \Theta_k$, respectively, and whose entries above the diagonal blocks are all equal to $\pi$. For brevity, we write $w = (w_1, \ldots, w_d)$ for the $d$-tuple given by $(w_1(\Theta), \ldots, w_d(\Theta))$. Since the entries of $\Theta$ above the diagonal blocks are all equal to $\pi$, it follows that $w_i$ and $w_j$ are anti-commuting whenever $i$ and $j$ correspond to different blocks.

    Let $\varepsilon>0$. By \cite[Theorem 4.4]{gerhold2021dilations}, for each $j=1,\ldots, k$, there is
    \[
        X_j\in\text{conv}\{ w_i : i \text{ corresponds to the $\Theta_j$ block}\}
    \]
    such that $\|\text{Re} X_j\| \leq c_{\Theta_j}^{-1} + \varepsilon$. For $i$ and $j$ belonging to different blocks, Fuglede's theorem implies that, since $w_i$ and $w_j$ anti-commute, the operators $\text{Re}X_i$ and $\text{Re}X_j$ anti-commute as well. Consequently, we have that
    \begin{align*}
        \left\|\sum_{j=1}^k \left( \frac{c_{\Theta_j}^2}{\sum_{\ell=1}^k c_{\Theta_\ell}^2} \right) \text{Re} X_j \right\|^2 & = \left\|\sum_{j=1}^k \left( \frac{c_{\Theta_j}^2}{\sum_{\ell=1}^k c_{\Theta_\ell}^2} \right)^2 (\text{Re} X_j)^2 \right\|\\
        & \leq \sum_{j=1}^k \left( \frac{c_{\Theta_j}^2}{\sum_{\ell=1}^k c_{\Theta_\ell}^2} \right)^2 \|\text{Re} X_j\|^2\\
        & \leq \sum_{j=1}^k \left( \frac{c_{\Theta_j}^2}{\sum_{\ell=1}^k c_{\Theta_\ell}^2} \right)^2 (c_{\Theta_j}^{-1} + \varepsilon)^2\\
        & = \left( \sum_{\ell=1}^k c_{\Theta_\ell}^2\right)^{-2} \left( \sum_{j=1}^k c_{\Theta_j}^4(c_{\Theta_j}^{-1}+\varepsilon)^2\right).
    \end{align*}
    So, we have that
    \[
        \inf_{X\in\conv(w)}\| \text{Re}X\|^2 \leq \left( \sum_{\ell=1}^k c_{\Theta_\ell}^2\right)^{-2} \left( \sum_{j=1}^k c_{\Theta_j}^4(c_{\Theta_j}^{-1}+\varepsilon)^2\right),
    \]
    and, as $\varepsilon>0$ was arbitrary, it follows that
    \[
        \inf_{X\in\conv(w)}\| \text{Re}X\|^2 \leq\left(\sum_{j=1}^k c_{\Theta_j}^2\right)^{-1}.
    \]
    Therefore, by \cite[Theorem 4.4]{gerhold2021dilations}, we may conclude that
    \[
        C_d\geq c_{\Theta} \geq \sqrt{\sum_{j=1}^k c_{\Theta_j}^2}.
    \]
\end{proof}

The aforementioned lower bound of $C_2 \gtrsim 1.5438$ is the value $c_{\theta_s}$ where $\theta_s= 2\pi(\sqrt{2}-1)$ is the so-called silver mean \cite[Example 7.3]{gerhold2022dilations}. As discussed within \cite{gerhold2022dilations}, a numerical computation shows that $c_{\theta_s}\approx 1.5437772$. Similarly, it was found in \cite[Corollary 4.7]{gerhold2021dilations} that $C_3\gtrsim 1.858$. Consequently, for all $d\geq 2$, Theorem \ref{t:uni-dil-lower} immediately implies that
\[
    C_d \gtrsim 1.0916\sqrt{d} \qquad \text{or} \qquad C_d \gtrsim \sqrt{1.1916d - 0.1227},
\]
depending on whether $d$ is even or odd, respectively. Indeed, the even case follows from decomposing $d$ into a sum of twos, whereas the odd case follows from decomposing $d$ into a sum of twos and a single three. This gives a strict improvement to the previous lower bound of $\sqrt{d}$ on $C_d$ for every $d\geq4$. Furthermore, for $d\geq 4$, this lower bound obstructs the possibility of improving the bounds on the Haar dilation constant $H_d$ from gaining direct insight on $C_d$, contrary to the circumstance for $d=2$ (Corollary \ref{c:haar-comm-dil-constant}). Indeed, using \cite[Theorem 3.10]{gerhold2021dilations}, one simply observes that
\[
    H_d \leq 2\sqrt{1-\frac{1}{2d}} < 1.0916\sqrt{d} \lesssim C_d, \qquad d\geq 4.
\]

\vspace{6pt}
\textbf{Acknowledgments.}
The author is grateful to Malte Gerhold, Orr Shalit, and to Adam Dor-On, for comments that have had a meaningful impact on this work. This material is based upon work supported by the Swedish Research Council under grant no.~ 2021-06594 while the author was in attendance at Institut Mittag-Leffler in Djursholm, Sweden during the program ``Operator Algebras and Quantum Information" in Spring of 2026.

\appendix
\section{Optimization of parameters for Theorem \ref{t:uni-comm-dil-constant}}\label{s:optimization}

This appendix provides information that justifies the choice of constants found in the proof of Theorem \ref{t:uni-comm-dil-constant}. This is intended to highlight the degree to which the parameters have been selected in an optimal way. As the strategy involved is somewhat technical, we provide an additional summary once all of the details have been recorded.

To this end, we first highlight the basic strategy of Theorem \ref{t:uni-comm-dil-constant}. Given a choice of constants $0<r<p<1$, we construct some unit vector $\xi = \sum_{m\in\bZ}\delta_m\otimes\xi_m$ in $\ell^2(\bZ)\otimes\K$ and a unitary operator $S\in B(\K)$ on some Hilbert space. By Proposition \ref{p:univ-uni-compress}, this gives us access to a unital completely positive map $\Phi_1:\rC(\bT^2)\rightarrow\rC^*(\bF_2)$ that is a slice of the commuting unitaries from Proposition \ref{p:cocyc-uni-comm-uni}. Recall that
\[
    \Phi_1(z) = \left( \sum_{m\in\bZ}\langle \xi_m, \xi_{m+1}\rangle \right)\mathfrak{u}_a \quad \text{and} \quad \Phi_1(w) = \sum_{m\in\bZ}\langle S\xi_m, \xi_m\rangle \mathfrak{u}_{a^mba^{-m}}.
\]
Upon averaging $\Phi_1$ across generators, we obtain the desired map $\Phi:\rC(\bT^2)\rightarrow\rC^*(\bF_2)$. The results of this appendix describe how the unit vector $\xi$ and the unitary $S$ from Proposition \ref{p:univ-uni-compress} are selected to be largely optimal, thereby justifying our choice of construction.

First note that, in the proof of Theorem \ref{t:uni-comm-dil-constant}, the vectors $\{\xi_m: m\neq 0\}$ are scalar multiples of a fixed unit vector $v\in\K$ satisfying $\langle Sv, v\rangle = 0$. The latter constraint is central to being able to meaningfully apply Proposition \ref{p:univ-uni-compress}, as this guarantees that $\Phi_1(w)$ is a scalar multiple of $\mathfrak{u}_b$. In which case, the resulting scalar coefficient of $\Phi_1(w)$ is $\langle S\xi_0, \xi_0\rangle$. The following describes how to select $S$, $\xi_0$, and $v$ so as to maximize $A = |\langle \xi_0, v\rangle|$, relative to the scalar values $r = |\langle S\xi_0, \xi_0\rangle|$ and $p= \|\xi_0\|^2$. This thereby maximizes the quantity $|\langle\xi_0,\xi_1\rangle|$ that appears in $\Phi_1(w)$.

\begin{proposition}\label{p:optimize-1}
    Let $S\in B(\K)$ be a unitary operator on some Hilbert space $\K$. Suppose that $v\in\K$ is a unit vector such that $\langle Sv, v\rangle = 0$. If $0<r<p<1$ are constants and $\eta\in\K$ is such that $|\langle S\eta, \eta\rangle| = r$ and $\|\eta\|^2 = p$, then
    \begin{equation}\label{eq:optimize-1}
        |\langle \eta, v\rangle| \leq \frac{\sqrt{p+r}+\sqrt{p-r}}{2}.
    \end{equation}
    Moreover, the upper bound is attained when $\K = \bC^2$, $S\in B(\K)$ is defined by $Se_1 = e_1$ and $Se_2 = -e_2$, $v = \frac{e_1+e_2}{\sqrt{2}}$, and $\eta = \sqrt{\frac{p+r}{2}}e_1 + \sqrt{\frac{p-r}{2}}e_2$.
\end{proposition}

\begin{proof}
    It is straightforward to see that the right-hand side of equation (\ref{eq:optimize-1}) is attained on $\K=\bC^2$ by the prescribed vectors $\eta, v$ and unitary $S$. We verify that this is indeed the maximal value for $|\langle\eta,v\rangle|$. For clarity, we consider the unit vector $\eta_0 = \frac{1}{\sqrt{p}}\eta$. Upon replacing $S$ by $zS$ for some $z\in\bT$, we may assume that
    \[
        \text{Re}(\langle S \eta_0, \eta_0\rangle) = \frac{r}{p}.
    \]
    In which case, we then have that
    \[
        \left\langle  (I+\text{Re}S)v, v \right\rangle = 1 \quad \text{and} \quad \left\langle (I+\text{Re}S)\eta_0, \eta_0 \right\rangle = \frac{p+r}{p}.
    \]
    Set $B=\frac{1}{2}(I+\text{Re}S)$, which is a positive contraction. Then, note that
    \begin{align*}
        |\langle\eta_0, v\rangle| & \leq |\langle B^{1/2} \eta_0,  B^{1/2}v\rangle| + |\langle (I-B)^{1/2}\eta_0, (I-B)^{1/2} v\rangle|\\
        & \leq \langle B\eta_0, \eta_0\rangle^{1/2} \langle Bv, v\rangle^{1/2} + \langle (I-B) \eta_0, \eta_0\rangle^{1/2}\langle (I-B)v, v\rangle^{1/2}\\
        & = \frac{\sqrt{p+r} + \sqrt{p-r}}{2\sqrt{p}}.
    \end{align*}
    From which, we have that $|\langle \eta, v\rangle| \leq \frac{1}{2}(\sqrt{p+r}+\sqrt{p-r})$.
\end{proof}

In the next step, we describe how the unit vector $\xi$ has been selected to be somewhat optimal from the given constants $p$ and $r$, as well as how $p$ and $r$ have been selected to produce an optimal output for $\Phi_1$. In accordance with Proposition \ref{p:optimize-1}, we assume the vectors $\{\xi_m: m\neq0\}$ to all be scalar multiples of the vector $v=\frac{e_1+e_2}{\sqrt{2}}$ in $\K=\bC^2$. In which case, we now maximize the scalar value that arises from $\Phi_1(z)$ relative to $q = \frac{1-p}{2}$ and
\[
     A= \langle \xi_0, v\rangle = \left\langle \sqrt{\frac{p+r}{2}}e_1 + \sqrt{\frac{p-r}{2}}e_2, \frac{e_1+e_2}{\sqrt{2}}\right\rangle = \frac{\sqrt{p+r}+\sqrt{p-r}}{2}.
\]
After the proof, we address the appearance of the constant $q$. We thank Malte Gerhold for the proof, which replaces a more complicated argument of the author.

\begin{proposition}\label{p:optimize-2}
    Fix constants $A,q>0$ and let $\F_q\subseteq\ell^2(\bN)$ be those $\alpha = (\alpha_m)_{m\geq 1}\in\ell^2(\bN)$ such that $\|\alpha\|^2 = q$. Define a function on $\F_q$ by
    \[
        T_A(\alpha) = A\alpha_1 + \sum_{m\geq1}\alpha_m\alpha_{m+1}.
    \]
    Then, we have $\sup_{\alpha\in\F_q}|T_A(\alpha)| = \sqrt{q(A^2+q)}$ with the supremum being attained at
    \[
        \alpha_m=\sqrt{\frac{A^2q}{A^2+q}}\,\left( \frac{q}{A^2+q}\right)^{\frac{m-1}{2}}, \qquad m\geq 1.
    \]
\end{proposition}

\begin{proof}
    First note that we may assume without loss of generality that $\alpha$ has non-negative entries. Then, let $S\in B(\ell^2(\bN))$ be the backward shift and observe that
    \[
        T_A(\alpha) = A\alpha_1 + \langle \alpha, S\alpha\rangle.
    \]
    By the Cauchy-Schwarz inequality, we have that
    \[
        T_A(\alpha) = A \alpha_1 + \langle \alpha, S\alpha\rangle \leq A \alpha_1 + \|\alpha\| \|S\alpha\| = A \alpha_1 + q\sqrt{1-\frac{\alpha_1^2}{q}},
    \]
    with equality whenever $\alpha$ is an eigenvector for $S$. When $\alpha$ is indeed an eigenvector for $S$, we necessarily have that $S\alpha = \rho\alpha$ where
    \[
        \rho = \sqrt{1-\frac{\alpha_1^2}{q}},
    \]
    and that $\alpha_m = \alpha_1 \rho^{m-1}$ for every $m\geq 1$. Note that $\begin{bmatrix}\frac{\alpha_1}{\sqrt{q}} & \rho\end{bmatrix}^T$ is a unit vector as $\alpha\in\F_q$. In turn, applying the Cauchy-Schwarz inequality again, we see that 
    \[
        T_A(\alpha)\leq A\alpha_1 + q\rho = \left\langle \begin{bmatrix}  A\sqrt{q}\\ q\end{bmatrix}, \begin{bmatrix}  \frac{\alpha_1}{\sqrt{q}}\\ \rho\end{bmatrix} \right\rangle \leq \sqrt{q(A^2+q)},
    \]
    with equality when
    \[
        \frac{\alpha_1}{\sqrt{q}} = \frac{A}{\sqrt{A^2+q}} \quad \text{and} \quad \rho = \sqrt{\frac{q}{A^2+q}}.
    \]
    It then follows that
    \[
        \alpha_1 = \sqrt{\frac{A^2q}{A^2+q}}.
    \]
    Therefore, we may conclude that
    \[
        \alpha_m = \alpha_1 \rho^{m-1}  = \sqrt{\frac{A^2q}{A^2+q}}\left( \frac{q}{A^2+q}\right)^{\frac{m-1}{2}}, \qquad m\geq 1.
    \]
\end{proof}

We remark that, in the proof of Theorem \ref{t:uni-comm-dil-constant}, the vectors $\{\xi_m: m\neq 0\}$ are assumed to satisfy $\xi_m = \xi_{-m}$ for each non-zero $m\in\bZ$. This is the setting where Proposition \ref{p:optimize-2} is immediately applicable. However, this assumption is of no loss for our purposes. Indeed, otherwise set $q_+ = \sum_{m\geq 1}|\alpha_m|^2$ and $q_- = \sum_{m\leq -1} |\alpha_m|^2$. By Proposition \ref{p:optimize-2}, the sum of the two quantities is bounded above by $\sqrt{q_+(A^2+q_+)}+\sqrt{q_-(A^2+q_-)}$. Now, the Cauchy-Schwarz inequality implies that
\[
    \left(\sqrt{q_+(A^2+q_+)}+\sqrt{q_-(A^2+q_-)}\right)^2 \leq (q_+ + q_-)\left((A^2+q_+) + (A^2+q_-)\right)
\]
with equality if $q_+ = q_-$. As $q_+ + q_- = 1-p$, the maximum then occurs when $q_+ = q_- = \frac{1-p}{2}$. Thus, it is sufficient to assume $\xi_m = \xi_{-m}$ for all non-zero $m\in\bZ$.

At this stage, we have now parametrized an optimal choice of unit vector $\xi$ and unitary $S$ in terms of $p$ and $r$. Indeed, the problem is reduced to maximizing the value of the final constant $\frac{r+\widetilde{r}}{2}$ in the proof of Theorem \ref{t:uni-comm-dil-constant}. This is because, in the proof of Theorem \ref{t:uni-comm-dil-constant}, it was found that
\[
    \widetilde{r} =  2\sqrt{q\left(A^2 + q \right)} = (1-p)\left(\frac{1-p}{2} + \left( \frac{\sqrt{p+r}+\sqrt{p-r}}{2} \right)^2 \right)^{1/2}.
\]
From this expression, we obtain that
\[
    \frac{1}{2}(r+\widetilde{r}) = \frac{1}{2}\left[r+\sqrt{(1-p)\bigl(1+\sqrt{p^2-r^2}\bigr)}\right].
\]
The final step is now to optimize this quantity over an appropriate choice of domain.

\begin{proposition}\label{p:optimize-3}
    For $0<r<p<1$, define
    \[
        t(p,r)=\frac{1}{2}\left[r+\sqrt{(1-p)\bigl(1+\sqrt{p^2-r^2}\bigr)}\right].
    \]
    Then, we have $\sup_{0<r<p<1} t(p,r)=\frac{\sqrt\phi}{2}$ with the supremum being attained at
    \[
        (p,r) =\left(\frac{5+\sqrt5}{10}, \sqrt{\frac{2(\sqrt5-1)}{5}}\right).
    \]
\end{proposition}

\begin{proof}
    For $0<s<p<1$, define
    \[
        T(p,s)=\sqrt{p^2-s^2}+\sqrt{(1-p)(1+s)}.
    \]
    Observe that
    \[
        \sup_{0<r<p<1} t(p,r) = \sup_{(p,s)\in D} \frac{1}{2}T(p,s)
    \]
    where we set $D=\{(p,s):0<s<p<1\}$. As such, we will show that
    \[
        \sup_{(p,s)\in D} T(p,s)=\sqrt{\frac{1+\sqrt5}{2}} = \sqrt{\phi}.
    \]

    Letting
    \[
        T_1(p,s)=\sqrt{p^2-s^2} \quad \text{and} \quad T_2(p,s)=\sqrt{(1-p)(1+s)},
    \]
    we have $T(p,s) = T_1(p,s) + T_2(p,s)$. Note that
    \[
        \left(\frac{\partial T}{\partial p},\frac{\partial T}{\partial s}\right)=\left(\frac{p}{T_1}-\frac{1+s}{2T_2}, -\frac{s}{T_1}+\frac{1-p}{2T_2}\right).
    \]
    For a critical point $(p,s)$ in the interior, we must have that
    \[
        \frac{p}{T_1}=\frac{1+s}{2T_2} \quad \text{and} \quad \frac{s}{T_1}=\frac{1-p}{2T_2}
    \]
    and that $p(1-p)=s(1+s)$. Upon solving the equations, one finds that $\frac{p}{s} = 2+\sqrt{5}$ on account of having that $0<s<p <1$. Whence, we see that
    \[
        (p_0, s_0) = \left(\frac{5+\sqrt{5}}{10}, \frac{3\sqrt{5}-5}{10}\right)
    \]
    is the unique critical point in the interior. Consequently,
    \[
        r_0 = \sqrt{p_0^2-s_0^2} = \sqrt{\frac{2(\sqrt{5}-1)}{5}}.
    \]
    and, upon evaluating against $(p_0,s_0)$, this reveals that
    \[
        T(p_0,s_0) = \sqrt{\frac{1+\sqrt{5}}{2}} = \sqrt{\phi}.
    \]

    It remains to check that the critical point in the interior is the global maximum. For this, note that $T$ extends continuously to the closure $\ol{D}=\{(p,s):0\leq s\leq p\leq1\}$. As such, it remains to consider the (topological) boundary of $\ol{D}$. Here, we arrive at a few cases:
    \begin{enumerate}[{\rm (i)}]
        \item If $p=0$, then $s=0$ and $T(0,0)= 1$.\vspace{1.5mm}
        \item If $p=1$, then $T(1,s) = \sqrt{1-s^2}\leq 1$.\vspace{1.5mm}
        \item If $s=p$, then $T(p,p) = \sqrt{1-p^2}\leq 1$.\vspace{1.5mm}
        \item If $s=0$, we find that $T(p,0) = p+\sqrt{1-p}\leq \frac{5}{4}$.\vspace{1.5mm}
    \end{enumerate}
    As $\sqrt{\phi}$ exceeds these values, the global maximum for $T$ is attained at the critical point $(p_0,s_0)$ and, consequently,
    \[
        \max_{0<r<p<1} t(p,r) = \frac{\sqrt{\phi}}{2}.
    \]
\end{proof}

We provide a final synopsis of the extent to which these arguments collectively demonstrate the strength of our choices in Theorem \ref{t:uni-comm-dil-constant}, given the construction obtained from Proposition \ref{p:univ-uni-compress}. To this end, consider a unitary operator $S\in B(\K)$ and a unit vector $\xi = \sum_{m\in\bZ}\delta_m\otimes\xi_m$ in $\ell^2(\bZ)\otimes\K$. From Proposition \ref{p:univ-uni-compress}, we obtain a unital completely positive map $\Phi_1:\rC(\bT^2)\rightarrow\rC^*(\bF_2)$ satisfying
\begin{equation}\label{eq:app}
    \Phi_1(z) = \left( \sum_{m\in\bZ}\langle \xi_m, \xi_{m+1}\rangle \right)\mathfrak{u}_a \quad \text{and} \quad \Phi_1(w) = \sum_{m\in\bZ}\langle S\xi_m, \xi_m\rangle \mathfrak{u}_{a^mba^{-m}}.
\end{equation}
To guarantee that $\Phi_1$ maps $w$ to a scalar multiple of $\mathfrak{u}_b$, we impose that $\langle S\xi_m, \xi_m\rangle = 0$ for every non-zero $m\in\bZ$ and, in addition, that $\{\xi_m: m\neq 0\}$ are scalar multiples of one another. Now, without loss of generality, we may assume that $\xi_m = \xi_{-m}$ for each non-zero $m\in\bZ$. Then, consider the constants $r=|\langle S\xi_0, \xi_0\rangle|$ and $p= \|\xi_0\|^2$, which may be assumed to satisfy $0<r < p< 1$. By Proposition \ref{p:optimize-1}, we find that the largest value of $|\langle \xi_0, v\rangle|$, where $v\in\K$ is a unit vector with $\langle Sv, v\rangle =0$, is found by selecting a two-dimensional diagonal matrix $S$ together with the vectors $v= \frac{e_1+e_2}{\sqrt{2}}$ and $\xi_0 = \sqrt{\frac{p+r}{2}}e_1 + \sqrt{\frac{p-r}{2}}e_2$. At this stage, we have that $\xi_m = \alpha_{|m|}v$ for some scalar values $\{\alpha_m : m\geq 1\}$. Proposition \ref{p:optimize-2} demonstrates how to select the scalar values $\{\alpha_m : m\geq 1\}$ so as to return an optimal result for the scalar value associated with $\Phi_1(z)$. The maximal output that can then be extracted from the average of the scalar values in equation (\ref{eq:app}) coincides with the solution to the optimization problem
\[
    \frac{1}{2}\left[ r+ \sqrt{(1-p)(1+\sqrt{p^2-r^2})} \right], \qquad 0<r<p<1.
\]
Finally, Proposition \ref{p:optimize-3} gives the answer to this problem.

\section{Asymptotic bounds on the dilation constant}\label{s:recursion}

The goal of this appendix is to complete the proof of Theorem \ref{t:asy-bound}, which gives an asymptotic upper bound on the universal commuting dilation constant. To this end, we define
\[
    A(p,r) = \sqrt{(1-p)(1+\sqrt{p^2-r^2})}, \qquad 0\leq r\leq p\leq 1,
\]
and set $M(r) = \sup_{r\leq p\leq 1} A(p,r)$ for $0\leq r\leq 1$. Then, we define a sequence
\[
    \gamma_{d+1} = \sup_{0\leq r \leq 1}\frac{M(r) + dr\gamma_d}{d+1}, \qquad \gamma_1=1.
\]
The goal of this appendix is to prove that
\[
    \frac{1}{\gamma_d} = \sqrt{2d-\log d + O(1)} \quad \quad \text{as $d\rightarrow\infty$}.
\]
We start with a couple observations.

\begin{lemma}\label{l:asy-1}
    The following statements hold.
    \begin{enumerate}[{\rm (i)}]
        \item We have that
        \[
            M(1-s) = s^{1/2}+\frac{1}{4}s^{3/2} + O(s^{5/2}) \quad \text{as $s\rightarrow0^+$}.
        \]
        \item We have that $\gamma_d^{-2} = O(d)$ as $d\rightarrow\infty$.
    \end{enumerate}
\end{lemma}

\begin{proof}
    (i): By substituting for $r=1-s$, $p=1-t$, and $u=s-t$, we observe that
    \begin{align*}
        M(1-s)^2 & = \sup_{0\leq t\leq s}t\left( 1 + \sqrt{(1-t)^2 - (1-s)^2} \right)\\
        & = \sup_{0\leq u\leq s}(s-u)\left(1+ \sqrt{u(2-2s+u)} \right).
    \end{align*}
    If we substitute $u=s^2v$ for some $v$, then this gives
    \[
        (s-u)\left(1+\sqrt{u(2-2s+u)}\right) = s + s^2\left(\sqrt{2v}-v\right)+O(s^3).
    \]
    Since the function $v\mapsto \sqrt{2v}-v$ attains its maximum uniquely at $v= \frac{1}{2}$, the maximizing point for $M(1-s)^2$ satisfies $u = \frac{1}{2}s^2 + O(s^3)$. Therefore,
    \[
        M(1-s)^2 = s + \frac{1}{2}s^2 + O(s^3)
    \]
    and so, the claim follows.

    (ii): First note that, for every $0\leq s\leq1$, we have
    \[
        M(1-s) = \sup_{1-s\leq p\leq 1} \sqrt{(1-p)\left(1+\sqrt{p^2-(1-s)^2}\right)} \geq \sqrt{s}
    \]
    upon selecting $p=1-s$.
    
    We now use an inductive argument to show that $\gamma_d^{-2} \leq 4d$ for every positive integer $d\geq1$. For this, assume that $\gamma_d^{-2}\leq 4d$.  In turn, upon selecting $s_d = \frac{\gamma_d^{-2}}{4d^2}$,
    \[
        \gamma_{d+1}  = \sup_{0\leq s\leq 1} \frac{M(1-s) + d(1-s)\gamma_d}{d+1}\geq \frac{\sqrt{s_d}+d(1-s_d)\gamma_d}{d+1} = \gamma_d \frac{d+\frac{\gamma_d^{-2}}{4d}}{d+1}.
    \]
    Since $\gamma_d^{-2}\leq 4d$, it now follows that
    \[
        \gamma_{d+1}^{-2}\leq \gamma_d^{-2}\left( \frac{d+1}{d+\frac{\gamma_d^{-2}}{4d}}\right)^2 \leq 4(d+1).
    \]
\end{proof}

Using Lemma \ref{l:asy-1}, we now extrapolate a finer expansion for $\gamma_d^{-2}$.

\begin{lemma}\label{l:asy-2}
    We have that
    \[
        \gamma_{d+1}^{-2} = \left( 1 + \frac{2}{d}+\frac{1}{d^2}\right)\gamma_d^{-2} - \left(\frac{1}{2d^2}+\frac{1}{d^3}\right)\gamma_d^{-4} + \frac{1}{8d^4}\gamma_d^{-6}+ O(d^{-2}) \quad \text{as $d\rightarrow\infty$}.
    \]
\end{lemma}

\begin{proof}
    Consider some $0\leq s\leq 1$ that maximizes the value of $M(1-s) + d(1-s)\gamma_d$. At $r=1$, the numerator in the recurrence for $\gamma_{d+1}$ gives
    \[
        M(1)+d\gamma_d = A(1,1) + d\gamma_d = d\gamma_d.
    \]
    Thus, we have $M(1-s) + d(1-s)\gamma_d \geq d\gamma_d$ and so, $M(1-s) \geq d\gamma_d s$. On the other hand, for every $1-s\leq p\leq 1$, we have that $\sqrt{p^2-(1-s)^2}\leq 1$ and so,
    \[
        M(1-s)^2  = \sup_{1-s\leq p\leq 1} (1-p)\left(1+ \sqrt{p^2-(1-s)^2}\right) \leq 2s.
    \]
    Lemma \ref{l:asy-1} (ii) then implies that $s\leq \frac{2\gamma_d^{-2}}{d^2} = O(d^{-1})$. Thus, upon setting $a = s^{1/2}$, we may apply Lemma \ref{l:asy-1} (i) to conclude that
    \begin{equation}\label{eq:num-exp}
        M(1-s) + d(1-s)\gamma_d = d\gamma_d + a -d\gamma_d a^2 + \frac{1}{4}a^3 +O(a^5).
    \end{equation}
    As $\gamma_d^{-2} = O(d)$ by Lemma \ref{l:asy-1} (ii), computing the maximum of equation (\ref{eq:num-exp}) reveals
    \[
        a = \frac{1}{2d\gamma_d} + O\left(\frac{1}{d^3\gamma_d^3}\right).
    \]
    Therefore, we have that
    \[
        \sup_{0\leq s\leq 1}\left( M(1-s) + d(1-s)\gamma_d \right) = d\gamma_d + \frac{1}{4d\gamma_d} + \frac{1}{32d^3\gamma_d^3} + O\left(\frac{1}{(d\gamma_d)^{5}}\right)
    \]
    and so,
    \[
        \frac{\gamma_{d+1}}{\gamma_d} = \frac{d+\frac{\gamma_d^{-2}}{4d} + \frac{\gamma_d^{-4}}{32d^3} + O\left(\frac{\gamma_d^{-6}}{d^5}\right)}{d+1}.
    \]
    As $\gamma_d^{-2} = O(d)$ by Lemma \ref{l:asy-1} (ii), the desired expansion now follows.
\end{proof}

Now, we prove the desired asymptotic expansion for $\gamma_d^{-1}$, which thereby justifies the asymptotic bound on the universal commuting dilation constant (Theorem \ref{t:asy-bound}).

\begin{proposition}\label{p:asy-3}
    We have that
    \[
        \frac{1}{\gamma_d} = \sqrt{2d-\log d + O(1)} \quad \quad \text{as $d\rightarrow\infty$.}
    \]
\end{proposition}

\begin{proof}
    Set $b_d = \gamma_d^{-2}$ and $x_d = \gamma_d^{-2}d^{-1}$. It now suffices to verify that $b_d = 2d-\log d+O(1)$. To this end, by Lemma \ref{l:asy-2}, we have the expansion
    \[
        b_{d+1}-b_d = 2x_d - \frac{1}{2}x_d^2 + \frac{1}{d}\left( x_d - x_d^2 + \frac{1}{8}x_d^3\right)  + O(d^{-2}).
    \]
    Thus, we see that
    \begin{equation}\label{eq:dyn}
        x_{d+1}-x_d = \frac{1}{d}\left( x_d - \frac{1}{2}x_d^2 \right) +O(d^{-2}).
    \end{equation}
    
    We claim that the difference equation given by (\ref{eq:dyn}) has $x_*=2$ as an attracting fixed point. Indeed, the function $G(x) = x- \frac{1}{2}x^2$ is positive on $(0,2)$, negative on $(2,\infty)$, and $G'(2)=-1$. In addition, note that $x_d = \gamma_d^{-2}d^{-1}\geq 1$ as $C_d\geq \sqrt{d}$ and that $(x_d)_d$ is a bounded sequence. Thus, for every $\delta>0$, there is some $\ol{y}>0$ such that $G(x) \geq \ol{y}$ for $x\in [1,2-\delta]$ and $G(x)\leq -\ol{y}$ for $x\in [2+\delta,\infty)$. In particular, by equation (\ref{eq:dyn}), for all sufficiently large $d$, we have that $x_{d+1}-x_d \geq \frac{\ol{y}}{2d}$ when $x_d\in[1,2-\delta]$ and $x_{d+1}-x_d \leq -\frac{\ol{y}}{2d}$ when $x_d\in[2+\delta, \infty)$. Now, since $(x_d)_{d\geq1}$ is bounded and the error term $O(d^{-2})$ in equation (\ref{eq:dyn}) is summable, it follows that every cluster point for $(x_d)_{d\geq1}$ must lie in $[2-\delta,2+\delta]$. As $\delta>0$ was arbitrary, $x_*=2$ is then the only possible cluster point. Therefore, $x_*=2$ is an attracting fixed point and we have $b_d = 2d+o(d)$.
    
    Express $x_d = 2+ \frac{\eta_d}{d}$ where $\eta_d = o(d)$. It now suffices to show that $\eta_d = -\log d+ O(1)$. For this, substitute $\eta_d$ into the expansion in Lemma \ref{l:asy-2} to obtain
    \begin{equation}\label{eq:boot-bound}
        \eta_{d+1} = \eta_d - \frac{1}{d} + O\left(\frac{1+\eta_d^2}{d^2}\right).
    \end{equation}
    We now verify that the error term is summable. For this, set $y_d = \frac{\eta_d}{d}$. As $G(2+y) = -y-\frac{1}{4}y^2$, equation (\ref{eq:dyn}) gives that $y_{d+1} = \left(1-\frac{1}{d} \right) y_d - \frac{1}{2d}y_d^2 + O(d^{-2})$. As $x_* = 2$ is an attracting fixed point, there are then constants $C>0$ and $\frac{1}{2}<c<1$ such that, for all sufficiently large $d$,
    \begin{equation}\label{eq:decay}
        \left| y_{d+1}\right| \leq \left( 1 - \frac{c}{d}\right)\left|y_d\right| + \frac{C}{d^2}.
    \end{equation}
    Since $c<1$, iterating equation (\ref{eq:decay}) reveals that $\frac{\eta_d}{d} = O(d^{-c})$. It then follows that $\eta_d = O(d^{1-c})$. As $c>\frac{1}{2}$, we have that $\frac{1+\eta_d^2}{d^2}= O(d^{-2})+O(d^{-2c})$ is summable. Therefore, summing over the recurrence for $(\eta_d)_{d\geq1}$ up to the $d$-th iterate, we obtain that $\eta_d = -\log d + O(1)$ and thus, $b_d = 2d+\eta_d = 2d-\log d + O(1)$.
\end{proof}

\bibliographystyle{plain}
\bibliography{bbl}

@book{agler2002pick,
  title={Pick interpolation and {H}ilbert function spaces},
  author={Agler, Jim and McCarthy, John and Kramer, Linus E},
  volume={44},
  year={2002},
  publisher={American Mathematical Soc.}
}

@article{ben2002tractable,
  title={On tractable approximations of uncertain linear matrix inequalities affected by interval uncertainty},
  author={Ben-Tal, Aharon and Nemirovski, Arkadi},
  journal={SIAM Journal on Optimization},
  volume={12},
  number={3},
  pages={811--833},
  year={2002},
  publisher={SIAM}
}

@article{bluhm2025inclusion,
  title={Inclusion constants for free spectrahedra with applications to quantum incompatibility},
  author={Bluhm, Andreas and Evert, Eric and Klep, Igor and Magron, Victor and Nechita, Ion},
  journal={arXiv preprint arXiv:2512.17706},
  year={2025}
}

@article{bluhm2023polytope,
  title={Polytope compatibility—{F}rom quantum measurements to magic squares},
  author={Bluhm, Andreas and Nechita, Ion and Schmidt, Simon},
  journal={Journal of Mathematical Physics},
  volume={64},
  number={12},
  year={2023},
  publisher={AIP Publishing}
}

@article{bluhm2022incompatibility,
  title={Incompatibility in General Probabilistic Theories, Generalized Spectrahedra, and Tensor Norms},
  author={Bluhm, Andreas and Jen{\v{c}}ov{\'a}, Anna and Nechita, Ion},
  journal={Communications in Mathematical Physics},
  volume={393},
  number={3},
  pages={1125--1198},
  year={2022},
  publisher={Springer}
}

@article{bluhm2020compatibility,
  title={Compatibility of quantum measurements and inclusion constants for the matrix jewel},
  author={Bluhm, Andreas and Nechita, Ion},
  journal={SIAM Journal on Applied Algebra and Geometry},
  volume={4},
  number={2},
  pages={255--296},
  year={2020},
  publisher={SIAM}
}

@article{bluhm2018joint,
  title={Joint measurability of quantum effects and the matrix diamond},
  author={Bluhm, Andreas and Nechita, Ion},
  journal={Journal of Mathematical Physics},
  volume={59},
  number={11},
  year={2018},
  publisher={AIP Publishing}
}

@book{brown2008textrm,
  title={{$C^{\ast} $}-Algebras and Finite-Dimensional Approximations},
  author={Brown, Nathanial Patrick and Ozawa, Narutaka},
  volume={88},
  year={2008},
  publisher={American Mathematical Soc.}
}

@inproceedings{chakraborty2019power,
  title={The Power of Block-Encoded Matrix Powers: Improved Regression Techniques via Faster {H}amiltonian Simulation},
  author={Chakraborty, Shantanav and Gily{\'e}n, Andr{\'a}s and Jeffery, Stacey},
  booktitle={46th {I}nternational {C}olloquium on {A}utomata, {L}anguages, and {P}rogramming ({I}{C}{A}{L}{P} 2019)},
  pages={33--1},
  year={2019},
  organization={Schloss Dagstuhl--Leibniz-Zentrum f{\"u}r Informatik}
}

@article{clouatre2025lifting,
  title={Lifting {S}ylvester equations: singular value decay for non-normal coefficients},
  author={Clou{\^a}tre, Rapha{\"e}l and Klippenstein, Brock and Slevinsky, Richard Mika{\"e}l},
  journal={Integral Equations and Operator Theory},
  volume={97},
  number={1},
  pages={7},
  year={2025},
  publisher={Springer}
}

@article{davidson2017dilations,
  title={Dilations, inclusions of matrix convex sets, and completely positive maps},
  author={Davidson, Kenneth R and Dor-On, Adam and Shalit, Orr Moshe and Solel, Baruch},
  journal={International Mathematics Research Notices},
  volume={2017},
  number={13},
  pages={4069--4130},
  year={2017},
  publisher={Oxford University Press}
}

@article{de2020quantum,
  title={Quantum magic squares: dilations and their limitations},
  author={De las Cuevas, Gemma and Drescher, Tom and Netzer, Tim},
  journal={Journal of Mathematical Physics},
  volume={61},
  number={11},
  year={2020},
  publisher={AIP Publishing}
}

@article{fritz2017spectrahedral,
  title={Spectrahedral containment and operator systems with finite-dimensional realization},
  author={Fritz, Tobias and Netzer, Tim and Thom, Andreas},
  journal={SIAM Journal on Applied Algebra and Geometry},
  volume={1},
  number={1},
  pages={556--574},
  year={2017},
  publisher={SIAM}
}

@article{gerhold2025empirical,
  title={Empirical bounds for commuting dilations of free unitaries and the universal commuting dilation constant},
  author={Gerhold, Malte and Scherer, Marcel and Shalit, Orr},
  journal={To appear in Journal of Experimental Mathematics, arXiv pre-print arXiv:2510.12540},
  year={2025}
}

@article{gerhold2021dilations,
  title={Dilations of unitary tuples},
  author={Gerhold, Malte and Pandey, Satish K and Shalit, Orr Moshe and Solel, Baruch},
  journal={Journal of the London Mathematical Society},
  volume={104},
  number={5},
  pages={2053--2081},
  year={2021},
  publisher={Wiley Online Library}
}

@article{gerhold2022dilations,
  title={Dilations of $q$-commuting unitaries},
  author={Gerhold, Malte and Moshe Shalit, Orr},
  journal={International Mathematics Research Notices},
  volume={2022},
  number={1},
  pages={63--88},
  year={2022},
  publisher={Oxford University Press}
}

@article{gerhold2024dilation,
  title={Dilation distance and the stability of ergodic commutation relations},
  author={Gerhold, Malte and Shalit, Orr},
  journal={arXiv preprint arXiv:2406.05864},
  year={2024}
}

@book{helton2019dilations,
  title={Dilations, linear matrix inequalities, the matrix cube problem and beta distributions},
  author={Helton, J and Klep, Igor and McCullough, Scott and Schweighofer, Markus},
  volume={257},
  number={1232},
  year={2019},
  publisher={American Mathematical Society}
}

@article{hu2020quantum,
  title={A quantum algorithm for evolving open quantum dynamics on quantum computing devices},
  author={Hu, Zixuan and Xia, Rongxin and Kais, Sabre},
  journal={Scientific reports},
  volume={10},
  number={1},
  pages={3301},
  year={2020},
  publisher={Nature Publishing Group UK London}
}

@article{li2024simple,
  title={Simple cycle reservoirs are universal},
  author={Li, Boyu and Fong, Robert Simon and Tino, Peter},
  journal={Journal of Machine Learning Research},
  volume={25},
  number={158},
  pages={1--28},
  year={2024}
}

@article{manvcinska2024constant,
  title={Constant-sized robust self-tests for states and measurements of unbounded dimension},
  author={Man{\v{c}}inska, Laura and Prakash, Jitendra and Schafhauser, Christopher},
  journal={Communications in Mathematical Physics},
  volume={405},
  number={9},
  pages={221},
  year={2024},
  publisher={Springer}
}

@book{nagy2010harmonic,
  title={Harmonic analysis of operators on {H}ilbert space},
  author={Sz.-Nagy, B{\'e}la and Foias, Ciprian and Bercovici, Hari and K{\'e}rchy, L{\'a}szl{\'o}},
  year={2010},
  publisher={Springer Science \& Business Media}
}

@article{passer2019shape,
  title={Shape, scale, and minimality of matrix ranges},
  author={Passer, Benjamin},
  journal={Transactions of the American Mathematical Society},
  volume={372},
  number={2},
  pages={1451--1484},
  year={2019}
}

@article{passer2018minimal,
  title={Minimal and maximal matrix convex sets},
  author={Passer, Benjamin and Shalit, Orr Moshe and Solel, Baruch},
  journal={Journal of Functional Analysis},
  volume={274},
  number={11},
  pages={3197--3253},
  year={2018},
  publisher={Elsevier}
}

@book{paulsen2002completely,
  title={Completely bounded maps and operator algebras},
  author={Paulsen, Vern},
  number={78},
  year={2002},
  publisher={Cambridge University Press}
}

@incollection{shalit2021dilation,
  title={Dilation theory: a guided tour},
  author={Shalit, Orr},
  booktitle={Operator theory, functional analysis and applications},
  pages={551--623},
  year={2021},
  publisher={Springer}
}

@article{stinespring1955positive,
  title={Positive functions on {$C^{\ast} $}-algebras},
  author={Stinespring, W Forrest},
  journal={Proceedings of the american mathematical society},
  volume={6},
  number={2},
  pages={211--216},
  year={1955},
  publisher={JSTOR}
}

@article{sznagy1954contractions,
  title={Sur les contractions de l'espace de {H}ilbert},
  author={Sz.-Nagy, B{\'e}la},
  journal={Acta Scientiarum Mathematicarum (Szeged)},
  volume={15},
  pages={87--92},
  year={1954}
}

\end{document}